\documentclass[12pt]{article}

\usepackage[all]{xypic}
\usepackage{amsmath}
\usepackage{amsfonts}
\usepackage{amssymb}
\usepackage{amsthm}
\usepackage{color}

\usepackage[top=2cm, bottom=2cm, left=2cm, right=2cm]{geometry}

\newtheorem{thm}{Theorem}[section]

\newtheorem{crl}[thm]{Corollary}

\newtheorem{prp}[thm]{Proposition}

\newtheorem{lmm}[thm]{Lemma}

\newtheorem{conj}[thm]{Conjecture}

\newtheorem{rmk}[thm]{Remark}

\newcommand {\mb}{\mathbb}
\newcommand {\Z}{\mb Z}
\newcommand {\R}{\mb R}
\newcommand {\C}{\mb C}
\newcommand {\Q}{\mb H}
\newcommand {\F}{\mb F}
\newcommand {\colim}{\textrm{colim}\ }

\newcommand {\emb}{\mathrm{Emb}}

\newcommand {\pri}{\mathrm{Prim}}
\newcommand {\ind}{\mathrm{Ind}}
\newcommand {\ex}{\mathrm{excess}}
\newcommand {\ext}{\mathrm{Ext}}
\newcommand {\lra}{\longrightarrow}
\newcommand {\la}{\langle}
\newcommand {\ra}{\rangle}

\begin{document}

\title{Filtered finiteness of the image of the unstable Hurewicz homomorphism with applications to bordism of immersions}

\author{Hadi Zare\\
        School of Mathematics, Statistics,
        and Computer Sciences\\
        College of Science, University of Tehran, Tehran, Iran\\
        \textit{email:hadi.zare} at \textit{ut.ac.ir}}
\date{}

\maketitle

\begin{abstract}
After recent work of Hill, Hopkins, and Ravenel on the Kervaire invariant one problem \cite{HHR}, as well as Adams' solution of the Hopf invariant one problem \cite{Adams-Hopfinv}, an immediate consequence of Curtis conjecture is that the set of spherical classes in $H_*Q_0S^0$ is finite. Similarly, Eccles conjecture, when specialised to $X=S^n$ with $n>0$, together with Adams' Hopf invariant one theorem, implies that the set of spherical classes in $H_*QS^n$ is finite. We prove a filtered version of the above the finiteness properties. We show that if $X$ is an arbitrary $CW$-complex such that $H_*X$ is finite dimensional then the image of the composition ${_2\pi_*}\Omega^l\Sigma^{l+2}X\to{_2\pi_*}Q\Sigma^2X\to H_*Q\Sigma^2X$
is finite; the finiteness remains valid if we formally replace $X$ with $S^{-1}$. As an immediate and interesting application, we observe that for any compact Lie group $G$ with $\dim\mathfrak{g}>1$, for any $n>0$ and $l>0$, the image of the composition $${_2\pi_*}\Omega^l\Sigma^{l+\dim\mathfrak{g}}BG_+^{[n]}\to{_2\pi_*}Q\Sigma^{\dim\mathfrak{g}}BG_+^{[n]}\to{_2\pi_*}Q\Sigma^{\dim\mathfrak{g}}BG_+\to {_2\pi_*}Q_0S^0\to H_*Q_0S^0$$
is finite where $\Sigma^{\dim\mathfrak{g}}BG_+\to S^0$ is a suitably twisted transfer map. Next, we consider work of Koschorke and Sanderson which using Thom-Pontrjagin construction provides a $1$-$1$ correspondence (a group isomorphism if $m+d>0$)
$\Phi^{N,\xi}_{m,d}:\mathrm{Imm}_\xi^d(\R^m\times N)\longrightarrow [N_+,\Omega^{m+d}\Sigma^dT(\xi)]$. We apply work of Asadi and Eccles on computing Stiefel-Whitney numbers of immersions to show that given a framed immersion $M\to\R^{n+k}$ and choosing $n$ very large with respect to $d$ and $k$, all self-intersection manifolds of an arbitrary element of $\mathrm{Imm}_\xi^d(\R^m\times N)$ are boundary. \\
These results are based on providing explicit upper bounds on the dimension of spherical classes in $H_*\Omega^lS^{l+1}$ as well as $H_*\Omega^l\Sigma^{l+2}X$ which will be achieved by eliminating higher powers of $2$ in $H_*(QS^1;\Z/2)$, respectively in $H_*(Q\Sigma^2X;\Z/2)$, from being spherical. This corrects proof of \cite[Lemma 3.6]{Zare-Els-1} in the case of $n=1$ and consequently shows that there exist no spherical class in $H_*\Omega^lS^{l+1}$ when $l\geqslant 3$ in dimensions $2^{l-1}l+3$ and above. This generalises \cite[Theorem 2.3]{Zare-Els-1} to the case $n=1$.
\end{abstract}

\textbf{AMS subject classification:$55Q45,55P42,55N22,57Q35,57R20$}\\
\textbf{Keywords:} Loop space, Dyer-Lashof algebra, Steenrod algebra, bordism of immersions, characteristic classes, Stiefel-Whitney numbers


\section{Introduction and statement of results}
For a pointed space $X$, let $QX=\colim\Omega^i\Sigma^iX$ be the infinite loop space associated to $\Sigma^\infty X$ and write $Q_0X$ for its base point component. Note that $\pi_i^sX\simeq\pi_iQX$ for $i\geqslant 0$, $\pi_i^sX\simeq\pi_iQ_0X$ for $i>0$, and if $X$ is path connected then $Q_0X=QX$. Curtis conjecture on spherical classes in $H_*QS^0$ reads as following.

\begin{conj}[Curtis conjecture]\label{Curtisconj} \cite[Theorem 7.1]{Curtis}
For $n>0$, only Hopf invariant one and Kervaire invariant one elements map nontrivially under the unstable Hurewicz homomorphism
$${\pi_n^s}\simeq{\pi_n}QS^0\to H_*QS^0.$$
\end{conj}

Throughout the paper, we work at the prime $2$ writing ${\pi_*^s}$ and ${\pi_*}$ for the $2$-component of homotopy and stable homotopy, and $H_*$ for $H_*(-;\Z/2)$. Since $Q$ is used for denoting operations (both upper and lower indexed) when appearing before a homology class as well as $\Omega^\infty\Sigma^\infty$ when appearing before a space, we then have diverted from the usual notation of writing $Q(-)$ for `the quotient module of indecomposable elements' functor and we instead write $\ind(-)$ for the latter. To be consistent in notation, we write $\pri(-)$ for the submodule of primitive elements functor.\\

A closely related conjecture is a conjecture due to Eccles which may be stated as follows.

\begin{conj}[Eccles conjecture] \label{Ecclesconj}
Let $X$ be a path connected $CW$-complex with finitely generated homology. For $n>0$, suppose $h(f)\neq 0$ where ${\pi_n^s}X\simeq{\pi_n}QX\to H_*QX$ is the unstable Hurewicz homomorphism. Then, the stable adjoint of $f$ either is detected by homology or is detected by a primary operation in its mapping cone.
\end{conj}

The conjecture is stated in very geometric terms. Note that the stable adjoint of $f$ being detected by homology means that $h(f)\in H_*QX$ is stably spherical, i.e. it survives under homology suspension $H_*QX\to H_*X$ induced by $\Sigma^\infty QX\to \Sigma^\infty X$ given by the adjoint of the identity $QX\to QX$. The conjectures are related in the following way.

\begin{prp}
The Eccles conjecture for $X=P$ implies Curtis conjecture. On the other hand, the Curtis conjecture implies Eccles conjecture for $X=S^n$ with $n>0$.
\end{prp}

The above statement should be well known, but we record a proof in the appendix. 
The above conjectures make predictions about the image of spherical classes under the unstable Hurewicz homomorphism $h:{\pi_*^s}X\simeq{\pi_*}QX\to H_*QX$. After the recent solution to the Kervaire invariant one problem \cite{HHR}, Curtis conjecture immediately implies that the image of ${\pi_*}Q_0S^0\to H_*Q_0S^0$ is finite. Also, Adams solution to Hopf invariant one problem together with the truth of Eccles conjecture would imply that the image of ${\pi_*}QX\to H_*QX$ is finite for $X=S^n$ with $n>0$. This motivates the following weak form of the above conjectures.

\begin{conj}[Weak geometric version of Curtis and Eccles conjectures]\label{weakconjectures}
(i) The image of the unstable Hurewicz homomorphism $h:{\pi_{*>0}}Q_0S^0\to H_*Q_0S^0$ is finite.\\
(ii) For a path connected space $X$, the image of the unstable Hurewicz homomorphism $h:{\pi_{*}}QX\to H_*QX$ is finite.
\end{conj}

Note that, there are other weakened formulations of Curtis conjecture such as ``weak conjecture on spherical classes'' due to Hung
\cite[Conjecture 1.3]{Hung-weakconjecture} (proved in \cite[Main Theorem]{HungNam}). Such conjectures are mainly stated in terms of the Lannes-Zarati homomorphism; although, they provide strong evidence but they don't immediately imply the conjecture \ref{Curtisconj} as noted in \cite{Hung-erratum}. We hope that geometric weak versions as above, allow a more accessible resolution of the above conjectures. In this note we prove a filtered version of the above conjectures. To put this in context, note that by definition $QX=\colim \Omega^i\Sigma^i X$ is filtered by spaces $\Omega^i\Sigma^i X$. 

\begin{thm}[Filtered Eccles conjecture(for double suspensions and $X=S^{-1}$)]\label{main1}
Suppose $X$ is a $CW$-complex of finite type such that for $x\in \widetilde{H}_*X$, $\dim x\leqslant k$. If $\xi\in H_*Q\Sigma^2X$ is spherical which arises from a class in ${\pi_*}\Omega^l\Sigma^{l+2}X$ then $\dim\xi\leqslant 2^l(k+2)+2^{l-1}(l-2)+2$. Consequently, the image of $h:{\pi_*}Q\Sigma^2X\to H_*Q\Sigma^2X$ when restricted to the image of the stabilisation map ${\pi_*}\Omega^l\Sigma^{l+2}X\to{\pi_*}Q\Sigma^2X$ is finite. The result is valid if we formally replace $X$ with $S^{-1}$; in particular, the image of the composition
$${\pi_*}\Omega^lS^{l+1}\to {\pi_*}QS^l\to H_*QS^l$$
is finite, with the composition vanishing in dimensions more than $2^{l-1}l+2$.
\end{thm}

Recall that the results of \cite[Theorem 2.2]{Zare-Els-1} and \cite[Theorem 1.11]{Zare-PEMS} deal with the cases of $0<l<4$ related to the case of $X=S^k$ with $k\geqslant -1$ of the above theorem where a complete verification of spherical classes in single, double, and triple loop spaces is achieved. Hence, the above theorem in the case of $X=S^{-1}$ completes proof of the finiteness property of the image of ${\pi_*}QS^1\to H_*QS^1$ when restricted to finite loop spaces $\Omega^lS^{l+1}$ with $l>0$.\\

An immediate corollary of the above Theorem, upon applying homology suspension is the following `filtered weak geometric version of Curtis conjecture (up to one suspension).

\begin{crl}\label{finitecurtisuptosuspension}
Suppose $\xi\in H_*Q_0S^0$ is a spherical class with $\sigma_*\xi\neq 0$ which arise from a class in $H_*\Omega^{l+1}S^{l+1}$. Then $\dim\xi\leqslant 2^{l-1}l+1$. Consequently, the image of the composition
$${\pi_*}\Omega^{l+1}S^{l+1}\to {\pi_*}Q_0S^0\to H_*Q_0S^0\to H_{*+1}QS^1$$
is finite, with the composition vanishing in dimensions more than $2^{l-1}l+1$.
\end{crl}

An important corollary of Theorem \ref{main1} is on the image of certain transfer maps. Let $G$ be a compact Lie group with Lie algebra $\mathfrak{g}$ on which $G$ acts through its adjoint representation; write $\mathrm{ad}_G=EG\times_G\mathfrak{g}\to BG$ for the adjoint bundle associated to this action. For the inclusion of a closed Lie group $H<G$ and a (virtual) vector bundle $\alpha\to BG$, the transfer map of Becker-Schultz-Mann-Miller-Miller is a stable map \cite{Milequi}, \cite{BeckerSchultz1} (see also \cite{KZproj}) $BG^{\mathrm{ad}_G\oplus\alpha}\to BH^{\mathrm{ad}_H\oplus\alpha|_{BH}}$ where $\alpha|_{BH}$ denotes restricted of $\mathrm{ad}_G$ to $BH$. We are interested in transfer maps associated to $1<G$ which after twisting with $-\mathrm{ad}_G$ yield a stable map $\Sigma^{\dim\mathfrak{g}}BG_+\to S^0$. The following now is immediate.

\begin{crl}\label{imageoftransfer}
Let $G$ be a compact Lie group with Lie algebra $\mathfrak{g}$ so that $\dim\mathfrak{g}>1$. Then, for any $n>0$ and $l>0$, the image of the composition
$${_2\pi_*}\Omega^l\Sigma^{l+\dim\mathfrak{g}}BG_+^{[n]}\to{_2\pi_*}Q\Sigma^{\dim\mathfrak{g}}BG_+^{[n]}\to{_2\pi_*}Q\Sigma^{\dim\mathfrak{g}}BG_+\to {_2\pi_*}Q_0S^0\to H_*Q_0S^0$$
is finite. Here, for a $CW$-complex $X$, we write $X^{[n]}$ for the $n$-th skeleton of $X$.
\end{crl}

The cases with $G=U(1)^{\times n}$ with $n>0$, $G=S^3$ are of special interest, as in homotopy one would be interested in framed manifolds with  $U(1)^{\times n}$-structures or $S^3$-structures. By Corollary \ref{imageoftransfer}, for any $i,j$ with $i+j>0$, the image of the composition
$${\pi_*}\Omega^l\Sigma^{l+2}(\C^iP_+\wedge\C P^j_+)\to{\pi_*}Q(\Sigma^2\C^iP_+\wedge\C P^j_+)\to {\pi_*}Q(\Sigma^2\C P_+\wedge\C P_+)\to {\pi_*}Q_0S^0\to H_*Q_0S^0$$
is finite. Similarly, for any $n>0$, the image of the composition
$${\pi_*}\Omega^l\Sigma^{l+3}\Q P^n_+\to{\pi_*}Q\Sigma^3\Q P^n_+\to{\pi_*}Q\Sigma^3\Q P_+\to{\pi_*}Q_0S^0\to H_*Q_0S^0$$
is finite.\\

Note that the finiteness of the image of the filtered unstable Hurewicz homomorphism Theorem \ref{main1} follows from the existence of explicit upper bounds on the dimension of spherical classes. The case of $X=S^{-1}$ of this theorem generalises \cite[Theorem 2.3]{Zare-Els-1} to the case $n=1$. The proof is almost analogous to that of \cite[Theorem 2.3]{Zare-Els-1} and will be based on eliminating higher powers of $2$ in $H_*\Omega^lS^{l+1}$ from begin spherical. We shall prove this latter by choosing a suitable basis for primitive elements in $H_*Q_0S^0$. We have the following.

\begin{thm}\label{2^t}
(i) Suppose $\zeta^2\in H_*QS^1$ is a spherical class, then $\zeta$ is odd dimensional.\\
(ii) Suppose $X$ is an arbitrary space of finite type and $\zeta^2\in H_*Q\Sigma^2X$ is spherical, then $\zeta$ is odd dimensional.
\end{thm}

Note that the map $\Omega^lS^{l+1}\to QS^1$ is the $l$-fold loop of the stabilisation map $S^{l+1}\to QS^{l+1}$, so in particular induces a multiplicative map. As the map in homology is injective, hence the above theorem also eliminates higher powers of $2$ in $H_*\Omega^lS^{l+1}$ from being spherical. We note that the proof of \cite[Lemma 3.6]{Zare-Els-1} in the case of $n=1$ is not correct as it is based on the false assumption that $[1]\in H_*QS^0$ is primitive which is not. Our result above then provides a correct proof for \cite[Lemma 3.6]{Zare-Els-1}.

As an application, we record one application to the study of `unstable' bordism groups of immersions \cite{AsadiEccles-determining}. Consider a $k$-dimensional vector bundle $\xi$ with $k>0$, a smooth manifold $N$ with $\dim(N)\geqslant 0$, and nonnegative integers $m,d$. The following is well known (see for example
\cite[Theorem 1.1]{KoschorkeSanderson}).

\begin{prp}\label{unstablebordismgroup}
Let $\xi$ be a $k$-dimensional vector bundle, $k\geqslant 0$, so that $T(\xi)$ is path connected. Let $\mathrm{Imm}_\xi^d(\R^m\times N)$ denote the set of bordism classes of triples $(M,f,\iota)$ where $f:M\to\R^m\times N$ is a codimension $k$ immersion with a $\xi$-structure on the normal bundle $\nu_f$, $\iota:M\to\R^d\times (\R^m\times N)$ an embedding with a splitting of its normal bundle $\nu_\iota\simeq\R^d\oplus\nu_f$. Then, there exists a $1$-$1$ correspondence (a group isomorphism if $m+d>0$)
$$\Phi^{N,\xi}_{m,d}:\mathrm{Imm}_\xi^d(\R^m\times N)\longrightarrow [N_+,\Omega^{m+d}\Sigma^dT(\xi)].$$
\end{prp}

Note that the groups $\mathrm{Imm}_\xi^d(\R^m\times N)$ stabilise as $d\to +\infty$ in an obvious manner. This allows to filter the bordism groups $\mathrm{Imm}_\xi(\R^m\times N)$ of \cite{AsadiEccles-determining} by $\mathrm{Imm}_\xi(\R^m\times N)=\colim_{d\to+\infty} \mathrm{Imm}_\xi^d(\R^m\times N)$. According to
\cite[Lemma 2.2, Theorem 2.4]{AsadiEccles-determining} the bordism class of a codimension $k$ immersion $f:M\to\R^{n+k}$ with a $\xi$-structure on its normal bundle, as well as the bordism class of its self-intersection manifolds, is determined 
by the unstable Hurewicz image of an element of $\pi_{n+k}QT(\xi)$ which corresponds to $f$ under $\Phi^{\R^{n+k},\xi}_{0,+\infty}$. As an application of Theorem \ref{main1} in the case of $X=S^n$ with $n>0$, we have the following.

\begin{thm}\label{SWnumbers1}
Suppose $f:M\to\R^{n+k}$ is a framed immersion which extends to an embedding $M\to\R^d\times\R^{n+k}$. If $n\ggg k,d$, ie $n$ is very large compared to $k$ and $d$ (in fact exponentially with respect to $d$) then all of self-intersection manifolds of the image of (the bordism class of) $f$ under the stabilisation map
$\mathrm{Imm}_{\R^k}^d(\R^{n+k})\to \mathrm{Imm}_{\R^k}(\R^{n+k})$ are boundary.
\end{thm}

We leave a more complete study of the effects of Theorem \ref{main1} to the theory of immersions to a future work.



\section{Preliminaries}

\subsection{Homology of $Q_0S^0$}\label{homologysection}
It is known that homology of spaces $QS^n$ and $\Omega^iS^{i+n}$, with $n\geqslant 0$, can be described using Kudo-Araki operation; the Browder operations and Cohen brackets vanish on $\Omega^iS^{i+n}$ when $n<+\infty$ (see for example \cite[Theorem 3]{KudoAraki},\cite[Theorem 7.1]{KudoAraki-Hn},
\cite[Page 86, Corollary 2]{DyerLashof}, \cite{CLM}). We wish to recall the descriptions in both lower indexed, and upper indexed operations; both of these descriptions are useful in proving some of our statements.

\textbf{Homology in terms of $Q_i$ operations.} For an $i$-fold loop space $X$, the operation $Q_j$ is defined for $0\leqslant j<i$ as an additive homomorphism
$$Q_j:H_*X\to H_{2*+j}X$$
so that $Q_0$ is the same as squaring with respect to the Pontrjagin product on $H_*X$ coming from the loop sum on $X$. The homology rings $H_*\Omega^i\Sigma^i S^n$ and $H_*QS^n$ when $n>0$, as algebras, can be described as
$$\begin{array}{lll}
H_*\Omega^i\Sigma^iS^n &\simeq &\Z/2[Q_{j_1}\cdots Q_{j_r}x_n:0<j_1\leqslant j\leqslant\cdots\leqslant j_r<i],\\
H_*QS^n                &\simeq &\Z/2[Q_{j_1}\cdots Q_{j_r}x_n:0<j_1\leqslant j\leqslant\cdots\leqslant j_r],
\end{array}$$
where, for $n>0$, $x_n\in\widetilde{H}_nS^n$ is a generator. Note that $Q_0Q_Jx_n=(Q_Jx_n)^2$ in the polynomial algebra, so this case is not included in the above description. In this description, we allow the empty sequence $\phi$ as a nondecreasing sequence of nonnegative integers with $Q_\phi$ acting as the identity; this in fact realises the monomorphism $H_*S^n\to H_*\Omega^i\Sigma^iS^n$ given by the inclusion of the bottom cell $S^n\to\Omega^i\Sigma^iS^n$ being adjoint to the identity $S^{i+n}\to S^{i+n}$.\\
For the case $n=0$, following Wellington \cite{Wellington} (see also the discussion \cite[Section 2]{Hunter} as well as \cite[Theorem 2.4]{Hunter}) we may describe $H_*\Omega^{i+1}_0S^{i+1}$ as follows. Let $M_iS^0$ be the free $\Z/2$-module generated by symbols $Q_J\iota_0$ with $J=(j_1,\ldots,j_r)$ a nonempty and nondecreasing sequence of nonnegative integers with $j_r<i+1$ so that $J\neq (0,0,\ldots,0)$. Write $[n]$ for the image of $n\in\pi_0\Omega^{i+1}S^{i+1}$ in $H_0\Omega^{i+1}S^{i+1}$ under the Hurewicz map. There is an embedding $M_iS^0\to H_*\Omega^{i+1}S^{i+1}$ which sends $Q_J\iota_0$ to $Q_J[1]*[-2^{l(J)}]$ where $l(J)=r$. We then have the following.
\begin{lmm}(Wellington)
For $i>0$, $H_*\Omega^{i+1}_0S^{i+1}$ is isomorphic to the free commutative algebra generated by $M_iS^0$ modulo the ideal generated by $\{Q_0x-x^2\}$.
\end{lmm}
This allows to think of $H_*\Omega^{i+1}_0S^{i+1}$ as
$$\Z/2[Q_{j_1}\cdots Q_{j_r}[1]*[-2^{l(J)}]:0< j_1\leqslant j\leqslant\cdots\leqslant j_r<i+1].$$
The empty sequence is excluded as it would give to many $0$-dimensional classes whereas $\Omega_0^{i+1}S^{i+1}$, being the base point component of $\Omega^{i+1}S^{i+1}$, is path connected. We also not allow the $Q_0$ as it is acting as the squaring operation which is already implicit in the polynomial algebra structure. For $Q_0S^0$ being the base point component of $QS^0$, we have
\cite[Page 86, Corollary 2]{DyerLashof} (see also \cite[Part I, Lemma 4.10]{CLM})
$$H_*Q_0S^0\simeq\Z/2[Q_J[1]*[-2^{l(J)}]:0< j_1\leqslant j\leqslant\cdots\leqslant j_r].$$
Note that for $J$ being nondecreasing in $Q_Jx$ is the same as $I$ being admissible when we translate $Q_Jx=Q^Ix$ in terms of upper indexed operations. Note that we have an obvious monomorphism of algebras $H_*\Omega^{i+1}_0S^{i+1}\to H_*Q_0S^0$ provided by the stabilisation map $\Omega^{i+1}_0S^{i+1}\to Q_0S^0$ \cite[Proposition 3.1]{Zare-Els-1}.\\

\textbf{Homology in terms of $Q^i$ operations.} For some purposes, it is easier to use description of $QX$ in terms of the operations $Q^i$. For an infinite loop space $Y$ these are additive homomorphisms $Q^i:H_*Y\to H_{*+i}Y$ operations which relate to lower indexed operations by $Q_iz=Q^{i+d}z$ for any $d$-dimensional homology class $z$. The description of homology in lower indexed operations, translated to the following
$$\begin{array}{lll}
H_*QS^n    &\simeq&\Z/2[Q^Ix_n:I\textrm{ admissible },\ex(I)>n],\\
H_*Q_0S^0  &\simeq&\Z/2[Q^I[1]*[-2^{l(I)}]:I\textrm{ admissible }],
\end{array}$$
where $I=(i_1,\ldots,i_s)$ is admissible if $i_j\leq 2i_{j+1}$, and $\ex(I)=i_1-(i+\cdots+i_s)$; in particular, note that for $Q^Ix_n=Q_Jx_n$, $\ex(I)=j_1$ and $I$ is admissible if $J$ is nondecreasing. For $S^n$ with $n>0$, we allow the empty sequence to be admissible with $\ex(\phi)=+\infty$ and $Q^\phi$ acting as the identity. In particular, in this description, $Q^iz=0$ if $i<\dim z$ and $Q^dz=z^2$ if $d=\dim z$.\\
The evaluation map $\Sigma QS^n\to QS^{n+1}$, adjoint to the identity $QS^n\to \Omega QS^{n+1}=QS^n$, induces the homology suspension $H_*QS^n\to H_{*+1}QS^{n+1}$. According to \cite[Page 47]{CLM} the homology suspension is characterised by the following properties: (1) $\sigma_*$ acts trivially on decomposable terms; (2) on the generators it is given by
$$\sigma_* Q^Ix_n=Q^Ix_{n+1}\textrm{ if }n>0,\ \sigma_*(Q^I[1]*[-2^{l(I)}])=Q^Ix_1\textrm{ if }n=0.$$
The following is useful in the process of eliminating some classes from being spherical.

\begin{lmm}\label{kernelofsuspension-0}
(i) For $0\leqslant i\leqslant +\infty$, the kernel of the homology suspension $\sigma_*:H_*\Omega^i_0S^{i}\to H_{*+1}\Omega^{i-1}S^{(i-1)+1}$ consists only of decomposable classes.
\end{lmm}

\begin{proof}
We prove the Lemma for $Q_0S^0$, that is the case of $i=+\infty$ and the case of $i<+\infty$ is similar. If $\xi\in H_*Q_0S^0$ represents an element of $\ind(H_*Q_0S^0)$ then $\xi=\sum Q^I[1]*[-2^{l(I)}]+D$ where the sum nonempty and is running over certain terms $\ex(I)>0$ and consequently $\ex(I)\geqslant 1$. Then,
$$\sigma_*\xi=\sum Q^Ix_1\neq 0.$$
Hence, only decomposable classes belong to the kernel of $\sigma_*$.
\end{proof}

The action of the Steenrod algebra on $H_*Q_0S^n$, $n\geqslant 0$, is determined by iterated application of Nishida relations which read as follows
$$Sq_*^rQ^a=\sum_{t\geqslant 0}{a-r\choose r-2t}Q^{a-r+t}Sq^{t}_*$$
Here, $Sq^r_*:H_*(-;\Z/2)\to H_{*-r}(-;\Z/2)$ is the operation induced by $Sq^r$ using the duality of vector spaces over $\Z/2$. Note that the least upper bound for $t$ so that the binomial coefficient could be nontrivial mod $2$, is $[r/2]+1$ (depending on the parity of $r$ this would be maximum value or maximum value of $t$ plus $1$). In particular, we have
$$Sq^1_*Q^{2d}=Q^{2d-1},\ Sq^{1}_*Q^{2t+1}=0.$$
We also have the Cartan formula $Sq^{2t}_*\zeta^2=(Sq^t_*\zeta)^2$ \cite{Wellington}.\\

\textbf{On the geometry of $QS^0$.} We write $Q_iS^0$ for the component $QS^0$ that correspond to stable map $S^0\to S^0$ of homological degree $i$. The composition of a map of degree $i$ induces a translation map $*[i]:Q_{-i}S^0\to Q_0S^0$ which satisfy $[m]*[n]=[m+n]$. In particular, each of these translations maps $*[i]:Q_iS^0\to Q_0S^0$ is a homotopy equivalence. We keep the usual notation of writing $*[i]$ for the map induced in homology. Moreover, applying $Q^I[1]$ to $[1]\in H_0Q_1S^1$ yields a class in $H_*Q_{2^{l(I)}}S^0$. For this reason, we need translation by $[-2^{l(I)}]$ in order to get a class in $H_*Q_0S^0$. Finally, note that the homotopy equivalence $*[i]$ is induced by a map of spaces, namely composition with a loop of degree $i$, so it induces an isomorphism of coalgebras in homology which in particular preserves primitive classes.

\subsection{Homology of $QY$ when $Y$ is path connected.}
Let $\{y_\alpha\}$ be an additive basis for the reduced homology $\widetilde{H}_*Y$. Then, as a module over the Dyer-Lashof algebra, we have
$$H_*QY\simeq\Z/2[Q^Iy_\alpha:I\textrm{ admissible, }\ex(I)>\dim y_\alpha].$$
Here, the notion/definition of admissibility and excess are as the same as Section \ref{homologysection}. The action of the $Sq^t_*$ operations on $H_*QY$ is described using Nishida relations introduced in section \ref{homologysection}.\\

The action of the homology suspension $\sigma_*:H_*QY\to H_{*+1}Q\Sigma Y$ is described by $\sigma_*Q^Iy_\alpha=Q^I\Sigma y_\alpha$. The same technique of Lemma \ref{kernelofsuspension-0} proves the following (which is of course well known).

\begin{lmm}\label{kernelofsuspension}
The kernel of $\sigma_*:H_*QY\to H_{*+1}Q\Sigma Y$ consists of only decomposable terms.
\end{lmm}

\subsection{Recollection on spherical classes}
A class $\xi\in\widetilde{H}_*X$ is called spherical if it is in the image of the Hurewicz homomorphism $h:\pi_*X\to H_*X$. A spherical class $\xi\in\widetilde{H}_*X$ has some basic properties: (1) it is primitive in the coalgebra $\widetilde{H}_*X$ {where the coalgebra structure is induced by the diagonal map $X\to X\times X$}, (2) it is $A$-annihilated, i.e. $Sq^i_*\xi=0$ for all $i>0$ where $Sq^i_*:\widetilde{H}_*X\to \widetilde{H}_{*-i}X$ is the operation induced by $Sq^i:\widetilde{H}^*X\to \widetilde{H}^{*+i}X$ by the vector space duality \cite[Lemma 2.5]{AsadiEccles}.\\

\begin{crl}\label{formofspherical-1}
(i) Let $Y$ be a path connected space. For $f:S^d\to Q\Sigma Y$ is any map with $h(f)\neq 0$ we have $h(f)=\sum Q^I\Sigma y_\alpha$ where $I$ is raging over certain admissible sequences $I$ with $\ex(I)\geqslant \dim y_\alpha+1$.\\
(ii) If $f:S^d\to QS^1$ is given with $h(f)\neq 0$ then $h(f)=\sigma_*h(f')=\sum Q^Ix_1$ where sum is running over certain terms with $\ex(I)\geqslant 1$.
\end{crl}

\begin{proof}
We prove (ii) and (i) is similar. Suppose $f:S^d\to QS^1$ is given with $f_*\neq 0$ and $d>1$. This means that the adjoint mapping $f':S^{d-1}\to Q_0S^0$ is also nontrivial in homology with $h(f)=\sigma_*h(f')$. By Lemma \ref{kernelofsuspension-0},
$$h(f')=\sum Q^I[1]*[-2^{l(I)}]+D$$
for some admissible sequences $I$ with $\ex(I)>0$ and $D$ being a sum of decomposable term. Note that the sum cannot be empty as otherwise $h(f)=\sigma_*h(f')=0$. Consequently,
$$h(f)=\sigma_*h(f')=\sum Q^Ix_1.$$
\end{proof}

We are interested in studying spherical classes in $H_*\Omega Y$ for some nice spaces $Y=QX$ with $X$ being of finite type. We work over $\Z/2$ which is a field over which we have K\"{u}nneth formula that turns $H_*\Omega Y$ into a Hopf algebra. The structure of Hopf algebras allows to make some early eliminations of spherical classes in $H_*\Omega Y$. Writing $\Omega_0Y$ for the component of $\Omega Y$ corresponding to the constant loop, $H_*\Omega_0Y$ is also a connected Hopf algebra. Note that if {$\pi_0\Omega Y\simeq G$ is a discrete group, then} on the level of spaces we have a weak equivalence $\Omega Y\simeq\Omega_0 Y\times G$. We recall some well known facts from \cite{Zare-Els-1}.

\begin{lmm}\label{2^t-1}
Suppose $Y$ is a space of finite type. If $\xi\in H_*\Omega^2_0Y$ is a decomposable spherical class, then $\xi=p^{2^t}$ for some primitive class $p\in H_*\Omega^2_0Y$ and $t>0$.
\end{lmm}

\begin{proof}
Since $\Omega^2_0Y$ is a double loop space, hence $H_*\Omega^2_0Y$ is a bicommutative Hopf algebra. By Milnor-Moore exact sequence \cite[Proposition 4.23]{MM} a decomposable primitive class must be square of a primitive class. A spherical class is in particular primitive. This completes proof.
\end{proof}

\section{Proof of Theorem \ref{2^t}}

\subsection{Primitive classes in $H_*Q_0S^0$}
The description that we give here, is essentially the one given by Madsen \cite[Proposition 5.3]{Madsen}. Note that we shall be working with primitive classes of $H_*Q_0S^0$ that live in odd dimensions. The main tool in computing primitive classes in a suitable Hopf algebra $H$ over any field $k$ of nonzero characteristic is the exact sequence of Milnor and Moore \cite[Proposition 4.23]{MM}. For $k=\F$, this gives
$$0\lra\pri(\F(sH))\lra\pri(H)\lra\ind(H)\lra\ind(\F(rH))\lra 0$$
where $\F(-)$ is the free $\F$-algebra functor, $\pri(-)$ and $\ind(-)$ are primitive and indecomposable functors, $s:H\to H$ is the squaring (a.k.a Frobenius map) and $r:H\to H$ is the square root map which is dual to the Frobenius map of the dual Hopf algebra $H^*$. Since, we are interested in the odd dimensional primitive classes, so the above sequence yield
$$0\lra\pri(H)\lra\ind(H)\lra\ind(\F(rH))\lra 0.$$
The dual of the Hopf algebra $H_*Q_0S^0$ is just $H^*Q_0S^0$ where on an $n$-dimensional class the squaring is just $Sq^n$. Consequently, in homology the squaring root map is given by the dual operations $Sq^r_*$. Hence, using Nishida relations, we see that the square root map $r:H_*Q_0S^0\to H_*Q_0S^0$ is described by $rQ^{2i+1}=0$ and $rQ^{2i}=Q^ir$ which extends over $H_*Q_0S^0$ by requiring $r$ to be linear and
$$rQ^iQ^I[1]*[-2^{l(I)+1}]=\left\{\begin{array}{ll}
                                  Q^krQ^I[1]*[-2^{l(I)+1}] & \textrm{if }i=2k,\\
                                  0                        & \textrm{if }i=2k+1.
                                  \end{array}\right.$$
The following is a triviality, but we record it for the reference.
\begin{lmm}\label{kernelofr}
If $r(\sum_{l(I)=s} Q^I[1]*[-2^{l(I)}])=0$ for some fixed $s>0$, then any sequence $I$ in the sum must have on odd entry.
\end{lmm}
\begin{proof}
By the above formulae, if $I$ has an odd entry $r(Q^I[1]*[-2^{l(I)}]=0$. Moreover, If $I$ and $J$ are two sequences with only even entries and of the same length, then $r(Q^I[1]*[-2^{l(I)}])\neq 0$ as well as $r(Q^I[1]*[-2^{l(I)}])+r(Q^J[1]*[-2^{l(J)}])\neq 0$. The lemma now follows.
\end{proof}

Now, we give a basis for the module of primitive in $H_*Q_0S^0$ over the Dyer-Lashof algebra. Note that by mixed Cartan formula, applying an operation $Q^i$ to a primitive class $\xi$, we have another primitive class $Q^i\xi$.
\begin{thm}\label{primitivebasis}
Consider all admissible sequences $I=(i_1,\ldots,i_s)$ with $l(I)=s>0$ such that $i_1$ is odd and $i_j$ is even for all $j>1$. For any such as sequence, let $p_I$ be the unique primitive element in $H_*Q_0S^0$ which module decomposable elements is given by
$$p_I=Q^I[1]*[-2^s],$$
i.e. $p_I=Q^I[1]*[-2^{l(I)}]+D_I$ where $D_I$ denotes a unique expression of decomposable elements. Then, up to translation by $*[-k]$ where $k$ is a difference of powers of $2$ and modulo decomposable terms, any primitive class in $H_*Q_0S^0$ is either of the form
$$\sum_{(I,J)\textrm{ admissible }}Q^Ip_J$$
or a power of $2$ of such classes.
\end{thm}

Note that this is slightly different from saying that $\pri(H_*Q_0S^0)$ is the module over the Dyer-Lashof algebra that is generated by $p_J$ elements. Here, $Q^Ip_J$ are ought to be admissible, and writing a primitive class in the claimed form follows naturally from the construction that is proposed in the following proof.

\begin{proof}
Note that by Milnor-Moore exact sequence above, a primitive class which is belongs to the submodule of decomposable elements, is square of a primitive class. Hence, we only need to describe those primitive classes that map nontrivially into the quotient module of indecomposable elements $\ind(H_*Q_0S^0)$ which are in $1$-$1$ correspondence with the kernel of the square root map. Suppose $\xi=\sum Q^I[1]*[-2^{l(I)}]+D$ be a primitive class where $D$ is the sum of decomposable terms and $\ex(I)>0$ for any terms in the first sum. By the above explanations, $\xi$ is primitive if and only if $\sum Q^I[1]*[-2^{l(I)}]\in\ker(r)$. Applying Lemma \ref{kernelofr} this means that $\xi$ is primitive if and only if any $I$ in the above expression has an odd entry. For any odd, let $b_I=\max\{j:i_j\textrm{ is odd}\}$ and split $I$ into two sequences as $I=(I',I'')$ with $I'=(i_1,\ldots,i_{b_I-1})$ and $I''=(i_{b_I},\ldots,i_s)$. Note that $I'=(\ )$, the empty sequence for $b_I=1$, and $l(I)=l(I')+l(I'')$. Now, we have
\begin{eqnarray*}
\xi&=&\sum Q^{I'}Q^{I''}[1]*[-2^{l(I)}]+D\\ \\
   &=&\sum Q^{I'}(Q^{I''}[1]*[-2^{l(I'')}])*[-2^{l(I)}+2^{l(I'')}]+D\\ \\
   &=&\sum Q^{I'}(p_{I''}+D_{I''})*[-2^{l(I)}+2^{l(I'')}]+D.
\end{eqnarray*}
Note that by Cartan formula $Q^{I'}D_{I''}$ is a sum of decomposable terms. Hence, modulo decomposable terms we have
$$\xi=\sum Q^{I'}p_{I''}*[-2^{l(I)}+2^{l(I'')}]$$
which is the desired equality. Also, note that $I=(I',I'')$ which is admissible by construction. This completes the proof.
\end{proof}

\subsection{Higher powers of $2$ in $H_*QS^1$}
Our description of primitive classes in $H_*Q_0S^0$ provided by Theorem \ref{primitivebasis} allows a straightforward proof of Theorem \ref{2^t}(i). We have the following.

\begin{lmm}\label{2^t-1}
If $\xi\in H_m\Omega^iS^{i+1}$, $0<i\leqslant +\infty$ is a decomposable spherical class. Then $\xi=\zeta^2$ for some odd dimension primitive class $\zeta$.
\end{lmm}

\begin{proof}
The stabilisation map $\Omega^iS^{i+n}\to QS^n$ is an $i$-fold loop map which induces an injection in homology \cite[Proposition 3.1]{Zare-Els-1}. Therefore, for $\xi=h(f)$ with $f:S^m\to\Omega^iS^{i+n}$, it is enough to prove the above lemma for $E^\infty f:S^m\to QS^n$ where by abuse of notation $E^\infty$ denotes the stabilisation map $\Omega^iS^{i+n}\to QS^n$. By Lemma \ref{2^t-1} if $\xi\in H_*QS^1$ is a decomposable spherical class (which is primitive in particular) then $\xi=\zeta^{2^t}$ for some $t>0$. By Corollary \ref{formofspherical-1} if $\xi=h(f)$ with $f:S^{2d}\to QS^1$ then
$$h(f)=(\sum Q^Ix_1)^2=Q^d(\sum Q^Ix_1)=\sum Q^dQ^Ix_1$$
where $d=\dim Q^Ix_1$ and sum is running over certain admissible sequences $I$. By Lemma \ref{kernelofsuspension-0}, for the adjoint map $f':S^{2d-1}\to Q_0S^0$ we have
$$h(f')=Q^d(\sum Q^I[1]*[-2^{l(I)+1}])+D_1=\sum Q^dQ^I[1]*[-2^{l(I)+1}]+D_1$$
where $D_1$ is a sum of decomposable classes. Since $h(f')$ is a primitive class, then the non-decomposable part of the above expression must belong to the kernel of the square root map. If $d$ is even then by Lemma \ref{kernelofr} each $I$ in the above expression for $h(f')$ must have at least one odd entry. Writing in terms of the notation that we used in the proof of Theorem \ref{primitivebasis}, we may rewrite $h(f')$ as
$$h(f')=\sum Q^dQ^{I'}Q^{I''}[1]*[-2^{l(I)+1}]+D_1=\sum Q^dQ^{I'}(p_{I''}+D_{I''})*[-2^{l(I)+1}+2^{l(I'')}]+D_1.$$
By Cartan formula, the part $Q^dQ^{I'}D_{I''}*[-2^{l(I)+1}+2^{l(I'')}]$ provides a sum of decomposable terms in $H_*Q_0S^0$. Hence, we have
$$h(f')=\sum Q^dQ^{I'}p_{I''}*[-2^{l(I)+1}+2^{l(I'')}]+D$$
where $D$ denotes a sum of decomposable terms. By diagonal Cartan formula \cite[Part I, Theorem 1.1(6)]{CLM}, applying an operation $Q^i$ to a primitive class yields another primitive class. Consequently, applying an iterated operation $Q^I$ to a primitive class yields a primitive class. Hence, any term $Q^dQ^{I'}p_{I''}$ in the above expression is primitive in $H_*Q_{2^{l(I')+1}}S^0$. The translations maps are isomorphisms of coalgebras, hence yield primitive classes. Therefore, the first sum in the above expression is primitive. As $h(f')$ is primitive. We deduce that $D$ is also primitive. As a sum of decomposable classes, $D$ must be a square but living in the odd dimension $2d-1$. Hence, $D=0$. Consequently,
$$h(f')=Q^d(\sum Q^{I'}p_{I''}*[-2^{l(I)+1}+2^{l(I'')}])$$
which upon applying $Sq^1_*$, using Nishida relation $Sq^1_*Q^d=Q^{d-1}$, yields
$$Sq^1_*h(f')=(\sum Q^{I'}p_{I''}*[-2^{l(I)+1}+2^{l(I'')}])^2\neq 0.$$
But, this contradicts the fact that $h(f')$ is $A$-annihilated. Hence, $d$ cannot be even.
\end{proof}

In particular, the above lemma implies that if $\xi=\zeta^{2^t}\in H_*\Omega^iS^{n+i}$ with $n>0$ is spherical then $t<2$.

\begin{rmk}
One may ask why we have not performed such a technique to eliminate higher powers of $2$ in $H_*Q_0S^0$ from being spherical. It is possible to do this, but then for the action of $Sq^1_*$ we obtain a square class in $H_*QS^{-1}$ which is an exterior algebra. For this reason this technique fails. It might be possible to adjoint such a map another time to land in $H_*QS^{-1}$ and then apply $Sq^1_*$. But, in this case we don't know whether the class the is obtained is nontrivial, and at the moment we have not resolved this case
\end{rmk}

\subsection{Higher powers of $2$ in $H_*Q\Sigma^2X$}
Recall some simple observations. First, note that for any space $Y$ the cohomology $H^*\Sigma Y$ does not have any nontrivial product. Consequently, any class in $H_*\Sigma Y$ is primitive. By diagonal Cartan formula \cite[Part I, Theorem 1.1(6)]{CLM}, if $\zeta\in H_*QZ$ is a primitive class so $Q^I\zeta$ is where $Z$ is an arbitrary space. This implies that all classes $Q^I\Sigma y_\alpha$ in $H_*Q\Sigma Y$ are primitive. Now, we are able to prove Theorem \ref{2^t}(ii) which completes proof of Theorem \ref{2^t}.

\begin{thm}\label{2^t-2}
Suppose $X$ is an arbitrary space of finite type and $\zeta^2\in H_*Q\Sigma^2X$ is spherical, then $\zeta$ is odd dimensional.
\end{thm}

\begin{proof}
Suppose $f:S^{2d}\to Q\Sigma^2X$ is given with $h(f)\neq 0$. For $Y=\Sigma X$, by Corollary \ref{formofspherical-1}, we may write
$$h(f)=(\sum Q^I\Sigma y_\alpha)^2=Q^d(\sum Q^Iy_\alpha)$$
where the sum is running over certain generators $I$, $d=\dim Q^I\Sigma y_\alpha$, and $\{y_\alpha\}$ is an additive basis for $\widetilde{H}_*Y$. Consequently, for the adjoint map $f':S^{2d-1}\to Q\Sigma X$ we have
$$h(f')=Q^d(\sum Q^Iy_\alpha)+D$$
where $D$ is a sum of decomposable elements. Since $Y$ is a suspension, so the terms of the first sum are primitive. Hence, $D$ is a primitive which is also decomposable. Hence, $D$ is a square term which lives in dimension $2d-1$. Therefore, $D=0$. If $d$ is even, then applying $Sq^1_*$ yields
$$Sq^1_*h(f')=(\sum Q^Iy_\alpha)^2\neq 0.$$
This contradicts the fact that $h(f')$ is $A$-annihilated. This completes the proof.
\end{proof}

\section{Proof of Main results}
We start with recalling some well known facts due to Wellington \cite{Wellington-thesis} modified to the case of $p=2$ and $n=+\infty$. Let $X$ be a path connected space. For the reduced homology $\widetilde{H}_*X$ of $X$, $S_\infty X$ be the set $\{Q^Ix:I\textrm{ admissible, }\ex(I)\geqslant\dim x\}$. Note that $H_*QX$ is just free algebra generated by $S_\infty X$ modulo the relation $Q^d\zeta=\zeta^2$ if $d=\dim\zeta$. Let $M_\infty X$ be the free $\Z/2$-vector space generated by $S_\infty X$ and let $M_\infty^{-}X$ be the submodule of $M_\infty X$ generated by all $I$ with $I$ having only odd entries. The obvious monomorphism $M_\infty X\to H_*QX$ respects the action of Steenrod operations, where $M_\infty X$ is forced to have an $A$-action through Nishida relations. Write $\mathrm{ann}(M_\infty X)$ for the submodule of $M_\infty X$ consisting of all $A$-annihilated elements of $M_\infty X$. The following is a part of \cite[Theorem 5.3]{Wellington-thesis} that will be used here.

\begin{thm}\label{Wellington5.3}
Suppose $y\in\mathrm{ann}(M_\infty X)$ with $\dim y$ being odd. Then, $y\in M_\infty^{-} X$.
\end{thm}

We wish to prove Theorem \ref{main1} which generalises \cite[Theorem 2.3]{Zare-Els-1} to the case $n=1$. Recall that according to \cite[Theorem 2.3]{Zare-Els-1} if $\xi\in H_*\Omega^lS^{n+l}$ is any class of dimension greater than $2^ln+2^{l-1}(l-2)+2$ with $n>1$ and $4\leqslant l\leqslant +\infty$ then it is not spherical. Here, we proof the theorem for $n=1$.
\begin{proof}[Proof of Theorem \ref{main1}]
\textbf{Case of $X=S^{-1}$.} Suppose $f:S^d\to\Omega^lS^{l+1}\to QS^1$ is given with $\xi=h(f)\neq 0$. If $\sigma_*\xi\neq 0$ then \cite[Theorem 2.3]{Zare-Els-1} provides an upper bound on $\dim\xi+1$ in terms of $l-1$ and $n+1$ which immediately yields the claimed upper bound on $\dim\xi$. Hence, suppose $\sigma_*\xi=0$. Then by Lemma \ref{2^t-1} and Theorem \ref{2^t}(i) $\xi=\zeta^2$ for some odd dimensional class $\zeta$. By \cite[Lemma 3.5]{Zare-Els-1}, $\xi$ can only be expressed as a sum of classes of the form $Q^Ix_1$, hence $\xi\in M_\infty S^1$. By Theorem \ref{Wellington5.3} (see also \cite[Theorem 5.3]{Wellington-thesis}), $\xi\in M_\infty^{-}S^1$ which means that written in upper indexed operations, if $Q^Ix_1$ is any term of $\zeta$ then it must only consist of odd entries. Equivalently, if $Q_Jx_1$ is any terms of $\zeta$ then it has to be strictly increasing. The maximum dimension of such $J$ would be achieved for $J=(1,2,\ldots,l-1)$. We then compute that
$$\begin{array}{lll}
\dim (Q_Jx_1) &    =     & 1+2(2+2(3+2(\cdots+(l-2+2((l-1)+2)))\cdots)\\
              &    =     & \sum_{i=0}^{l-2}2^{i-1}i+2^{l-1}\\
              &    =     & 2^{l-2}((l-1)-1)+1+2^{l-1}
              \end{array}$$
where we have used the equality $\sum_{i=0}^k2^{i-1}i=2^k(k-1)+1$.
Consequently,
$$\dim h(f)\leqslant 2(2^{l-2}((l-1)-1)+1+2^{l-1})=2^{l-1}l+2.$$
This implies that there are no spherical classes above this dimension.\\
\textbf{Case of $\Sigma^2X$.} The proof is similar to the previous case and is based on induction. If $\sigma_*h(f)\neq 0$ then $\sigma_*h(f)$ is a spherical class in $H_{i+1}Q\Sigma^3X$. Applying the inductive hypothesis, we have a bound on $i+1$ which immediately gives the desired bound on $i$. Therefore, assume $\sigma_*h(f)=0$. By Lemma \ref{kernelofsuspension}, $h(f)$ is a decomposable class in $H_*Q\Sigma Y$ with $Y=\Sigma X$. By Lemma \ref{2^t-1} and Theorem \ref{2^2-2}, $h(f)=\zeta^2$ for an odd dimensional $A$-annihilated primitive class $\zeta\in H_*Q\Sigma Y$. By Corollary \ref{formofspherical-1}, if $f:S^d\to Q\Sigma Y$ is given with $h(f)\neq 0$, then $h(f)$ is a sum of certain terms of the form $Q^I\Sigma y_\alpha$ which as $h(f)$ is a square means that $\ex(Q^I\Sigma y_\alpha)=0$ for any term involved in $h(f)$. Consequently, by Lemma \ref{Wellington5.3}, for $h(f)=(\sum Q^{I_1}\Sigma y_\alpha)^2$ with $I_1=(i,\ldots,i_s)$ the class $I_1$ will consist of only odd entries. Also note that, writing in lower indexed operations, $Q^{I_1}\Sigma y_\alpha=Q_J\Sigma y_\alpha$, for $I_1$ having only odd entries is the same as $j_k+j_{k+1}$ being odd \cite[Lemma 7.1]{Zare-Els-1}. Also recall that $I_1$ admissible is the same as $J$ being nondecreasing. Together with $j_k+j_{k+1}$ being odd for $k=1,\ldots,s-1$ if $J=(j_1,\ldots,j_s)$, this latter implies that $J$ has to be strictly increasing, i.e. $j_1<j<\cdots<j_s$.\\
Next, note that the stabilisation map $E^\infty:\Omega^l\Sigma^{l+2}X\to Q\Sigma^2X$ in homology sends classes of the form $Q^I\Sigma y_\alpha$ identically where $\{y_\alpha\}$ is a homogeneous basis for $\widetilde{H}_*Y$ with $Y=\Sigma X$. Consequently, if $f_1:S^d\to\Omega^l\Sigma^{l+2}X$ is any pull back of $f$, then modulo $\ker(E^\infty_*)$, $h(f_1)\in H_*\Omega^l\Sigma^{l+2}X$ precisely will consist of the same terms $Q^I\Sigma y_\alpha$ that belong to $h(f)$.\\
Now we sum up. Suppose $h(f)=(Q_J\Sigma y_\alpha)^2$ modulo other terms where $J=(j_1,\ldots,j_s)$ has to be a strictly increasing sequence with $j_k+j_{k+1}$ being odd. Our aim is to find the maximum possible dimension for such a class. If $H_*X$ has its top class in dimension $n$, then $H_*\Sigma Y=H_*\Sigma^2X$ has its top class in dimension $n+2$. The maximum dimension for $Q_J\Sigma y_\alpha$ will be achieved if we choose $J=(1,2,\ldots,l-1)$ and $\Sigma y_\alpha$ is a top dimensional class. Let $\zeta_{n+2}\in H_*\Sigma Y$ be a class of dimension $n+2$. We can compute that
$$\begin{array}{lll}
\dim (Q_J\zeta_{n+2}) &    =     & 1+2(2+2(3+2(\cdots+(l-2+2((l-1)+2(n+2))))\cdots)\\
                &    =     & \sum_{i=0}^{l-2}2^{i-1}i+2^{l-1}(n+2)\\
                &    =     & 2^{l-2}((l-1)-1)+1+2^{l-1}(n+2)
              \end{array}$$
where we have used the equality $\sum_{i=0}^k2^{i-1}i=2^k(k-1)+1$.
Consequently,
$$\dim h(f)\leqslant 2(2^{l-2}((l-1)-1)+1+2^{l-1}(n+2))=2^l(n+2)+2^{l-1}(l-2)+2.$$
This provides the claimed upper bound.
\end{proof}

\section{Applications to bordism of immersions}
We start with fixing some notation and conventions. For a $k$-dimensional vector bundle $\xi\to X$ with $X$ being compact, let $f:X\to BO(k)$ be the classifying map. For a subgroup $G<O(k)$, we say $\xi$ has a $G$-structure if the classifying map admits a lift $f_G:X\to BG$ whose composition with $BG\to BO(k)$ equals to $f$. This is that same as saying $\xi=f_G^*\gamma^k|_G$ with $\gamma^k\to BO(k)$ being the universal bundle whose restriction over $BG$ we denote by $\gamma^k|_G$. We abuse the notation by writing $\R^n$ for both the Euclidean space and the $n$-dimensional trivial bundle over an arbitrary base space.\\

\textbf{Bordism theory of immersions.} The theory of bordism for immersions arises from bordism theory of embeddings. For an ambient $n$-dimensional manifold $N$, the notion of bordism between two embeddings $i_j:M_j\to N$ of compact manifolds $M_i$ with $i=0,1$ of the same codimension with the normal bundle of $v_{i_j}$ having a $\xi$-structure is understood. Through the Thom-Pontrjagin construction, it yields a $1$-$1$ correspondence
$$\Phi^{N,\xi}:\emb_\xi(N)\to[N_+,T\xi]$$
where $\emb_\xi(N)$ is the set of bordism classes of embeddings into $N$ with a $\xi$-structure, and $T(-)$ is the Thomification functor. Next, let $f:M\to N$ be a compact manifold with a codimension $k$ immersion into $N$ so that the normal bundle $\nu_f$ has a $\xi$-structure. The compactness of $M$ implies that there exists an embedding $i:M\to\R^d$ for some $d$. This allows to make the immersion $M\stackrel{f}{\to} N\stackrel{(1,*)}{\to} N\times\R^d$ isotopic to the embedding $(f,i):M\to N\times\R^d$. Moreover, we may define the bordism relation between to immersions of $M$ into $N$ in the usual way (see \cite{AsadiEccles-determining}). An application of Thom-Pontrjagin construction to $(1,*)\circ f$, or equivalently to $(f,i)$, provides us with a map $(\R^d\times N)_+\to M^{\nu_f\oplus\R^d}$ whose adjoint is a map $N_+\to\Omega^d\Sigma^d M^{\nu_f}$ which we may compose with $\Omega^d\Sigma^d T\xi$. Note that for an immersion, it is the stable normal bundle which does not depend on the immersion which means that one has to let $d\to +\infty$. To go in the reverse direction, one may apply the multi-compression theorem of Rourke and Sanderson \cite[Theorem 4.5]{RourkeSanderson-Comperssion.I} (see also \cite[Theorem 2.2]{EcclesGrant}).

\begin{thm}[The Multi-Compression Theorem]
Suppose $M^{n-k}\to N^n\times\R^l$ is an embedding with $l$ linearly independent normal vector fields so that $\nu_f\simeq\nu\oplus\R^l$ for some $k$-dimensional vector bundle $\nu$. If $k>0$ then $g$ is isotopic to an embedding $g':M\to N\times\R^l$ so that the composition
$$f:M\stackrel{g'}{\longrightarrow} N\times\R^l\longrightarrow N$$
is an immersion with $\nu_f\simeq\nu$. Here, the second arrow on right is just projection map.
\end{thm}

The outcome of this procedure, is an isomorphism
$$\Phi^{N,\xi}:\mathrm{Imm}_\xi(N)\to [N_+,QT\xi]$$
of groups provided by the Pontrjagin-Thom construction. Here $\mathrm{Imm}_\xi(N)$ is the group of bordism class of immersions into $N$ with a $\xi$-structure on their normal bundle. We note that, before letting $d\to +\infty$, we may define intermediate bordism groups, which we call unstable bordism groups, denoted by $\mathrm{Imm}_\xi^d(\R^m\times N)$ as the set of bordism classes (bordism group if $m+d>0$ which we shortly call unstable bordism group) of triples $(M,f,\iota)$ with $f:M\to\R^m\times N$ being a codimension $k$-immersion, $\iota:M\to\R^d\times (\R^m\times N)$ an embedding with a splitting of its normal bundle $\nu_\iota\simeq\R^d\oplus\nu_f$. The following is due to Koschorke and Sanderson \cite[Theorem 1.1]{KoschorkeSanderson}.
\begin{prp}\label{unstablebordismgroups-1}
Suppose $T\xi$ is path connected. Then, there is an isomorphism of sets (of groups if $m+d>0$)
$$\Phi^{N,\xi}_{m,d}:\mathrm{Imm}_\xi^d(\R^m\times N)\longrightarrow [N_+,\Omega^{m+d}\Sigma^dT(\xi)].$$
\end{prp}

Note that there is an obvious procedure of stablisation the above bordism sets allowing $d\to +\infty$ which allows to define $\mathrm{Imm}_\xi(\R^m\times N)=\colim_{d\to +\infty}\mathrm{Imm}_\xi^d(\R^m\times N)$. Note that $T\xi$ is always path connected unless $\xi$ is a $0$-dimensional bundle, the case in which $f$ amounts to a finite cover of $N$. The application that we wish to present here, corresponds to the case of $m=0$, $\xi=\R^k$, and $N=\R^{n+k}$. For any codimension $k$ immersion $f$, then $\nu_f$ has an obvious $\gamma^k$-structure. According to Asadi and Eccles \cite[Lemma 2.2, Theorem 2.4]{AsadiEccles-determining}, for an immersion $f:M\to\R^{n+k}$, the normal Stiefel-Whitney numbers of $M$ are computed by the relation
$$\la w^Iw_k,h^s(\alpha)\ra$$
where $\alpha\in\pi_{n+k}QMO(k)$ is the element corresponding to the bordism class of $(M,f)$ under $\Phi^{\R^{n+k},\gamma^k}$, and $h^s:\pi_{n+k}QMO(k)\to H_{n+k}QMO(k)\to H_{n+k}MO(k)$ is the stable Hurewicz homomorphism. Moreover, the Stiefel-Whitney numbers of the $r$-fold intersection manifold of $f$ are determined by
$$\la w^Iw_{kr},(\xi_r)_*(p_r)_*h(\alpha)\ra$$
where $\xi_r:D_rMO(k)\to MO(kr)$ is induced by the forgetful map $B(O(k)\wr\Sigma_r)\to BO(k)$, and $p_r:\Sigma^\infty QX\to \Sigma^\infty D_rX$ is the stable projection map provided by the Snaith splitting. In particular, $p_1:\Sigma^\infty QX\to \Sigma^\infty X$ is the evaluation map which relates $h$ and $h^s$ with $h^s=(p_1)_*h$. Moreover, $\xi_1$ is the identity map, which shows that in the case of $r=1$ the first formula is derived from the second one. Note that, it is possible to have $h(\alpha)\neq0$ but $h^s(\alpha)=0$. This implies that it is possible to have an immersion bordant to a boundary, but some of its self-intersection manifolds are not. 
Now, we are able to prove Theorem \ref{SWnumbers1}.

\begin{proof}[Proof of Theorem \ref{SWnumbers1}]
By construction, there is a commutative diagram
$$\xymatrix{
\mathrm{Imm}_\xi^d(\R^m\times N)\ar[r]\ar[d]     & \mathrm{Imm}_\xi(\R^m\times N)\ar[d]\\
[N_+,\Omega^{m+d}\Sigma^dT(\xi)]\ar[d]^-h\ar[r]  & [N_+,\Omega^mQT(\xi)]\ar[d]^-h\\
H_t\Omega^{m+d}\Sigma^dT(\xi)\ar[r]              & H_t\Omega^{m}QT(\xi)
}$$
where $t=\dim N$ and the Hurewicz homomorphisms (downward arrows from second to third row) are defined by $h(f)=f_*[N]$ with $[N]$ being the fundamental class of $N$. Moreover, if the normal bundle of an immersion $f:M\to N$, classified by a map $M\to BO(k)$, admits a trivialisation, then the classifying map must factor as $M\to B1\to BO(k)$ which after Thomification yields a map $M^{\nu_f}\to S^k\to MO(k)$. For $m=0$, $\xi=\gamma^k$, $N=\R^{n+k}$, and $f:M\to\R^{n+k}$ a codimension $k$ immersion with a trivialisation of its normal bundle which also can be made isotopic to an embedding $M\to\R^d\times\R^{n+k}$, may extend the above diagram to a commutative diagram as
$$\xymatrix{
\mathrm{Imm}_{\R^k}^d(\R^{n+k})\ar[r]\ar[d]   & \mathrm{Imm}_{\gamma^k}^d(\R^{n+k})\ar[r]\ar[d]  & \mathrm{Imm}_{\gamma^k}(\R^{n+k})\ar[d]\\
\pi_{n+k}\Omega^{d}S^{d+k}\ar[r]\ar[d]^-h     & \pi_{n+k}\Omega^{d}\Sigma^dMO(k)\ar[d]^-h\ar[r]  & \pi_{n+k}QMO(k)\ar[d]^-h\\
H_{n+k}\Omega^{d}S^{d+k}\ar[r]                & H_{n+k}\Omega^{d}\Sigma^dMO(k)\ar[r]             & H_{n+k}QMO(k).
}$$
We wish to compute Hurewicz image, the composition of downward arrow on right, of those elements that belong to the image of $\mathrm{Imm}_{\gamma^k}^d(\R^{n+k})\to\mathrm{Imm}_{\gamma^k}(\R^{n+k})$ which also admit a trivialisation, that is we wish to compute the Hurewicz image of the element falling into the image of the composition given by the first row; ie we wish to compute the composition of arrows on the first row and right column. The theorem now follows from commutativity of the above diagram and choose $d$, $k$, and $n$ in a suitable manner. More precisely, choose $d=l$, $k$, and $n$ as in Theorem \ref{main1}. Consequently, in this range, the composition of downward arrows on the left is trivial. This would imply that $h(\alpha)=0$ for any $\alpha$ corresponding to a suitable chosen immersion as above. The formula of Asadi and Eccles, quoted above, implies that all of Stiefel-Whitney numbers of such immersion, as well as the Stiefel-Whitney numbers of the self-transverse immersions, are trivial. Since, we are dealing with immersion into Euclidean spaces, Thom's theory then proves the claim.
\end{proof}

\section{Discussion}
We wish to focus on the implications of Theorem \ref{main1} in the special case of $X=S^n$ with $n>0$. For the ease of notation, we drop the pre-subindex $2$ from ${\pi_*}$. Suppose $f\in{\pi_d}QS^n\simeq{\pi_{d-n}^s}$, with $d>2^ln+2^{l-1}(l-2)+2$. For $l\geqslant 1$, we have
$$d>2^ln+2^{l-1}(l-2)+2>2n+(l-2).$$
On the other hand, for the isomorphism $\pi_{d}\Omega^lS^{n+l}\simeq\pi_{d+l}S^{n+l}\simeq\pi_{d-n}^s$ to hold, we need $d+l<2(n+l-1)$ which yields $d<2n+l-2$. This means that the upper bound provided in this paper, always forces us to fall beyond the stable range. Moreover, the exponential factor, is more likely to force us to land in a highly unstable range. However, with our current knowledge on the groups $\pi_iS^k$ and the behaviour of the suspension maps $\pi_iS^k\to \pi_{i+j}S^{k+j}$, Theorem \ref{main1} is far from a triviality. Note that the elements of $\pi_d\Omega^lS^{l+n}$ with $d$ chosen as above, map to $\pi_{d}QS^n\simeq\pi_{d-n}^s$ after $d$ times suspension (the least number of suspensions required is $d-2n-l+3$). A possible scenario in which our result would appear as a triviality is to have a positive answer to a question such as\\

\textbf{Question.} (i) Suppose $d>2^ln+2^{l-1}(l-2)+2$ such that $\pi_{d-n}^s\not\simeq 0$. Is the image of the suspension map $E^{d-2n-l+3}:\pi_d\Omega^lS^{l+n}\to\pi_{d-n}^s$ trivial?\\
(ii) Suppose $\pi_{d-n}^s\not\simeq 0$. Does there exist $k>0$ so that for $d>k$, the image of $\pi_{d}\Omega^lS^{n+l}\to\pi_{d-n}^s$ is trivial?\\

To the author's knowledge, there is no answer to this question or questions like this in the literature. \\

\appendix

\section{On the relation between Curtis and Eccles conjectures}
This part is mostly expository and well known. We wish to record some observations on the relations between two conjectures. Below, we shall write $P$ for the infinite dimensional real projective space, $P^n$ for the $n$-dimensional real projective space, and $P_n=P/P^{n-1}$ for the truncated projective space with its bottom cell at dimension $n$. We also write $X_n$ for the truncated of a $CW$-complex where all skeleta of dimension $<n$ are collapsed to a point, i.e. $X_n=X/X^n$ where $X^n$ is the $n$-skeleton of $X$.

\begin{lmm}
Curtis conjecture, implies Eccles conjecture for $S^k$ with $k>0$.
\end{lmm}

\begin{proof}
Suppose $f:S^{n+k}\to QS^k$ is given with $h(f)\neq 0$. By adjointing down $k$ times, we have a map, say $f':S^n\to QS^0$, so that $\sigma_*^kh(f')=h(f)$ where $\sigma_*^k:H_*Q_0S^0\to H_{*+k}QS^k$ is the $k$-fold iterated suspension homomorphism. In particular, $h(f')\neq 0$ in $H_*Q_0S^0$. Assuming Curtis conjecture, $f'$ must be either a Hopf invariant or a Kervaire invariant one element. It is well known (see for example \cite[Theorem 7.3]{Madsenthesis}) that the unstable Hurewicz image of a Kervaire invariant one element, if it exists, in $H_*Q_0S^0$ is square of a certain primitive class, say $p_{2^i-1}^2$. However, decomposable classes are killed by homology suspension. So, $f'$ and consequently $f$, as elements of ${\pi_*^s}$ can only be Hopf invariant one elements, that is detected by a primary operation in its mapping cone. This prove Eccles conjecture for $S^k$.
\end{proof}

On the other hand, Curtis conjecture can be deduced from Eccles conjecture, thanks to Kahn-Priddy theorem, and its algebraic version due to Lin.

\begin{lmm}
Eccles conjecture for $X=P$ implies Curtis conjecture.
\end{lmm}

\begin{proof}
Consider the Kahn-Priddy map $\lambda:QP\to Q_0S^0$ which is an infinite loop map, inducing an epimorphism on ${\pi_*}$ on positive degrees \cite[Theorem 3.1]{Kahn-Priddy}, as well as an epimorphism on the level of Adams spectral sequences $\ext_A^{s,t}(H^*P,\Z/2)\to \ext_A^{s+1,t+1}(\Z/2,\Z/2)$ where $A$ denotes the mod $2$ Steenrod algebra \cite[Theorem 1.1]{Lin}. Suppose $f\in{\pi_n}Q_0S^0$ with $h(f)\neq 0$. Let $g\in{\pi_n}QP$ be any pull back of $f$ through $\lambda$. Then $g$ maps nontrivially under the unstable Hurewicz map ${\pi_n}QP\to H_nQP$. Assuming Eccles conjecture, implies that the stable adjoint of $g$ is detected either by homology or a primary operation in its mapping cone.\\
If the stable adjoint of $g$ is detected by homology, then it is detected on the $0$-line of the Adams spectral sequence for $P$. By Lin's result, $f=\lambda g$ is detected on the $1$-line of the Adams spectral sequence. The $1$-line of the Adams spectral sequence is known to detect Hopf invariant one elements, i.e. the stable adjoint of $f$ is a Hopf invariant one element.\\
Next, suppose the stable adjoint of $g$ is detected by a primary operation in its mapping cone. That is for some $i,j$ we have $Sq^ia_j=g_{n+1}$ in $C_{g'}$ where we write $g':S^n\to P$ for the stable adjoint of $g$. From the action of Steenrod algebra on $H^*P$ and decomposition of Steenrod squares to operations of the form $Sq^{2^t}$, for such an equation to hold in $C_{g'}$ we need $i=2^s$ and $j=2^t$ for some $s,t\geqslant 0$. From this, we see that $g'$ has to be detected in the $1$-line of the Adams spectra sequence for $P$. By Lin's theorem, this means that the stable adjoint of $f$ has to be detected in the $2$-line of the Adams spectral sequence. It is known that \cite{Za-ideal} the only elements on the $2$-line of the Adams spectral sequence that map nontrivially under $h$ are the Kervaire invariant one element, i.e. $f$ can only be a Kervaire invariant one element. This completes the proof.
\end{proof}

\bibliographystyle{plain}


\end{document}